\documentclass[11pt,reqno]{amsart}

\usepackage[a4paper,margin=31mm]{geometry}
\usepackage[T1]{fontenc}
\usepackage{lmodern}
\usepackage{microtype}
\usepackage{amsmath,amssymb,mathtools}
\usepackage{enumitem}
\usepackage{xcolor}
\usepackage[numbers,sort&compress]{natbib}
\usepackage[colorlinks=true,linkcolor=blue!55!black,citecolor=blue!55!black, urlcolor=blue!55!black]{hyperref}

\numberwithin{equation}{section}
\theoremstyle{plain}
\newtheorem{theorem}{Theorem}[section]
\newtheorem{proposition}[theorem]{Proposition}
\newtheorem{lemma}[theorem]{Lemma}
\newtheorem{corollary}[theorem]{Corollary}
\theoremstyle{definition}
\newtheorem{definition}[theorem]{Definition}
\theoremstyle{remark}
\newtheorem{remark}[theorem]{Remark}

\newcommand{\R}{\mathbb R}
\newcommand{\B}{\mathcal B}
\newcommand{\Sch}{\mathcal S}
\newcommand{\supp}{\operatorname{supp}}
\newcommand{\dd}{\,\mathrm d}
\newcommand{\norm}[2][]{\lVert #2\rVert_{#1}}
\newcommand{\abs}[1]{\lvert #1\rvert}
\newcommand{\Bmix}[3]{\mathcal B^{#1}_{#2;#3}}

\title[Sharp mixed spectral Barron regularity] {Sharp Mixed Spectral Barron Regularity of Coulombic Many-Electron Wave Functions}

\author[P. Ming and H. Yu]{Pingbing Ming and Hao Yu}
\address{SKLMS, Institute of Computational Mathematics and Scientific/Engineering Computing, Academy of Mathematics and Systems Science, Chinese Academy of Sciences, Beijing 100190, China}
\address{School of Mathematical Sciences, University of Chinese Academy of Sciences, Beijing 100049, China}
\email{mpb@lsec.cc.ac.cn}
\email{yuhao@amss.ac.cn}

\subjclass[2020]{Primary 35B65, 35J10; Secondary 35Q40, 46E35, 81V55}
\keywords{Coulomb Hamiltonian, electronic wave function, spectral Barron space, mixed regularity, Pauli principle, antisymmetry}

\begin{document}

\begin{abstract}
  We establish sharp mixed spectral Barron regularity for eigenfunctions of molecular Coulomb Hamiltonians. The mixed norm is a Fourier $L^1$ norm with one isotropic weight and coordinate-product weights, and therefore detects regularity invisible to the isotropic Barron scale. For a nonempty set $I$ of electron indices on which the wave function is antisymmetric, we derive an explicit admissible region for the isotropic order $s$ and the coordinate orders $\alpha,\beta$. This region is optimal as a uniform statement over the class of clamped-nuclei Coulomb Hamiltonians. For fixed-spin components with two occupied spin blocks, it reduces to $s+\alpha+\beta<1$; in the fully spin-polarized class it reduces to $s+\alpha<1$. In particular, if $\mathcal I_\sigma$ denotes the family of occupied same-spin blocks determined by $\sigma$, then every fixed-spin spatial component $\psi_\sigma$ satisfies, for every $0\leq\alpha<1$, \[ \left(\sum_{I\in\mathcal I_\sigma}\prod_{i\in I}\langle\xi_i\rangle^\alpha\right)\widehat{\psi_\sigma}\in L^1(\mathbb{R}^{3N}). \] For a fully spin-polarized state, $\mathcal I_\sigma=\{\{1,\ldots,N\}\}$.

\end{abstract}

\maketitle

\section{Introduction}

Let $N\geq1$ electrons move in the field of $L\geq1$ clamped nuclei with positive charges $Z_1,\dots,Z_L$ at distinct positions $R_1,\dots,R_L\in\R^3$. In atomic units, the Hamiltonian~is \begin{equation} \label{eq:Hamiltonian} H=-\frac12\sum_{i=1}^N\Delta_{x_i}+V,\qquad V(x)=-\sum_{i=1}^N\sum_{\nu=1}^L\frac{Z_\nu}{\abs{x_i-R_\nu}} +\sum_{1\leq i<j\leq N}\frac1{\abs{x_i-x_j}}. \end{equation}
The operator $H$ denotes the self-adjoint Friedrichs realization of the differential expression in \eqref{eq:Hamiltonian}. The Coulomb singularities occur in three-dimensional collision variables even though the configuration space has dimension $3N$. This structure is responsible for the mixed regularity of electronic wave functions and, at the same time, for the failure of direct isotropic estimates to describe all available smoothness.

The classical mixed-Sobolev theory began with Yserentant's work and its subsequent refinements. Kreusler and Yserentant proved that general electronic eigenfunctions have fractional product regularity of every order below $3/4$ \cite[Theorem~5.1]{KreuslerYserentant2012}; their model in Section~6 shows that this threshold is sharp in that scale. Yserentant proved full first-order mixed-Sobolev regularity for an explicitly cusp-regularized eigenfunction \cite[Theorem~6.6]{Yserentant2011}. More recently, Meng used the Pauli principle to prove sharper spin-dependent mixed-Sobolev estimates \cite{Meng2023}.

Spectral Barron regularity is a different, Fourier $L^1$ notion. It controls absolute Fourier moments rather than square-integrable derivatives. Yserentant proved that Coulombic electronic eigenfunctions belong to the isotropic spectral Barron space $\B^s$ for every $s<1$ \cite[Theorem~4.7]{Yserentant2026}; the hydrogen ground state shows that $s=1$ cannot be reached in general. A general Fourier--Lebesgue theory for many-particle Schr\"odinger eigenfunctions was developed in our earlier work \cite[Section~2.2]{MingYu2025}. The present paper refines these isotropic conclusions by retaining coordinate-product weights and by using cancellation forced by the Pauli principle.

Our main contributions are as follows.
\begin{enumerate}[label=(\roman*),leftmargin=2.2em]
  \item We introduce mixed spectral Barron spaces with an isotropic weight and coordinate-product weights. We prove index-set-dependent mixed Barron regularity for Coulombic eigenfunctions and, for fixed-spin components, simultaneous regularity with respect to every occupied spin block.
  \item We prove that the resulting mixed-regularity regions are uniformly sharp over the class of clamped-nuclei Coulomb Hamiltonians. Hydrogen, helium, and lithium ground states show that none of the defining inequalities of these parameter regions can be relaxed uniformly over this class.
\end{enumerate}

The paper is organized as follows. Section~\ref{sec:setting} defines the spaces and states the main results. Section~\ref{sec:operators} proves the weighted Coulomb estimates. Section~\ref{sec:gain} derives the mixed regularity gain of the resolvent. Section~\ref{sec:eigenfunction} completes the proof by a contraction argument. Section~\ref{sec:consequences} contains the proofs of the sharpness results. Section~\ref{sec:conclusion} summarizes the results and identifies directions for further work. Appendix~\ref{app:comparison-proofs} contains the auxiliary space-comparison results and their proofs.

\section{Setting and main results}\label{sec:setting}

For $f\in\Sch(\R^d)$, we use the unitary Fourier transform \[ \widehat f(\xi)=(2\pi)^{-d/2}\int_{\R^d}e^{-ix\cdot\xi}f(x)\,\dd x \] and extend it to $\Sch'(\R^d)$ by duality. Write $\langle z\rangle=(1+\abs z^2)^{1/2}$. For $\xi=(\xi_1,\dots,\xi_N)\in(\R^3)^N$ and a nonempty set $I\subseteq\{1,\dots,N\}$, define \begin{equation} \label{eq:weight} w_{s,I}^{\alpha,\beta}(\xi) =\langle\xi\rangle^s \prod_{i\in I}\langle\xi_i\rangle^\alpha \prod_{j\notin I}\langle\xi_j\rangle^\beta. \end{equation}

\begin{definition}[Mixed spectral Barron space]\label{def:mixedBarron}
  For $s,\alpha,\beta\geq0$, define \[ \Bmix{s}{I}{\alpha,\beta}(\R^{3N}) \coloneqq\left\{f\in\Sch'(\R^{3N})\,\middle|\, w_{s,I}^{\alpha,\beta}\widehat f\in L^1(\R^{3N})\right\}, \] equipped with the norm \[ \norm[\Bmix{s}{I}{\alpha,\beta}]{f} \coloneqq\int_{\R^{3N}}w_{s,I}^{\alpha,\beta}(\xi) \abs{\widehat f(\xi)}\,\dd\xi. \]
\end{definition}

For $\alpha=\beta=0$, this is the usual spectral Barron space $\B^s$. Unlike a single isotropic power, the product weights record simultaneous growth in several electron momenta.

We say that $u$ is antisymmetric with respect to $I$ if \begin{equation*} u(\dots,x_i,\dots,x_j,\dots)=-u(\dots,x_j,\dots,x_i,\dots) \quad (i,j\in I,\ i\neq j). \end{equation*}
For $\abs{I}=1$ this condition is void.

\begin{theorem}[Index-set-dependent mixed Barron regularity]
  \label{thm:mainI} Let $\psi\in H^1(\R^{3N})\setminus\{0\}$ satisfy $H\psi=E\psi$ for some $E\in\R$, and suppose that $\psi$ is antisymmetric with respect to a nonempty set $I\subseteq\{1,\dots,N\}$. Let $s,\alpha,\beta\geq0$.
  If
  \begin{equation}\label{eq:mainconditions}
    \begin{cases}
      s+\alpha<1,&I^c=\varnothing,\\ s+\alpha+\beta<1,&\abs{I^c}=1,\\ s+\max\{\alpha+\beta,2\beta\}<1,&\abs{I^c}\geq2.
    \end{cases}
  \end{equation}
  Then $\psi\in\Bmix{s}{I}{\alpha,\beta}(\R^{3N})$.
\end{theorem}

The family of parameter regions~\eqref{eq:mainconditions}, with $N$ and the nonempty set $I$ allowed to vary, is uniformly optimal over the class of clamped-nuclei Coulomb Hamiltonians; see Corollary~\ref{cor:uniformsharpness}.

For the spin formulation, let $q\in\mathbb N$ be the number of spin states, let $\Psi(x,\sigma)$ be a fermionic electronic eigenfunction, and fix $\sigma=(\sigma_1,\dots,\sigma_N)\in\{1,\dots,q\}^N$. Define \begin{equation*} \psi_\sigma(x)=\Psi(x,\sigma),\qquad I_\ell(\sigma)=\{i\mid\sigma_i=\ell\},\qquad \mathcal I_\sigma=\{I_\ell(\sigma)\mid I_\ell(\sigma)\neq\varnothing\}. \end{equation*}
For distinct $i,j$, let $U_{ij}$ exchange $x_i$ and $x_j$. If $i,j\in I_\ell(\sigma)$, the corresponding spin transposition leaves $\sigma$ unchanged, and the Pauli principle gives
\begin{equation}\label{eq:spin-antisymmetry} U_{ij}\psi_\sigma=-\psi_\sigma \qquad(i,j\in I_\ell(\sigma),\ i\neq j). \end{equation}
Thus $\psi_\sigma$ is antisymmetric with respect to every occupied spin block $I\in\mathcal I_\sigma$; see also \cite[(1.8)--(1.9) and Remark~1.3]{Meng2023}.
\begin{theorem}[Spin-partition refinement]\label{thm:spinrefined}
  Let $\psi=\psi_\sigma$ be a fixed-spin spatial component of an electronic eigenfunction, and let $s,\alpha,\beta\geq0$. If
  \begin{equation}\label{eq:spinconditions}
    \begin{cases}
      s+\alpha<1,&\abs{\mathcal I_\sigma}=1,\\ s+\alpha+\beta<1,&\abs{\mathcal I_\sigma}=2,\\ s+\max\{\alpha+\beta,2\beta\}<1,&\abs{\mathcal I_\sigma}\geq3.
    \end{cases}
  \end{equation}
  Then
  \begin{equation}\label{eq:spinintersection} \psi\in\bigcap_{I\in\mathcal I_\sigma}\Bmix{s}{I}{\alpha,\beta}(\R^{3N}). \end{equation}
\end{theorem}

For two spin states, Lemma~\ref{lem:minimal-two-spin-space} shows that, at fixed total order $s+\alpha+\beta<1$, the conclusion of Theorem~\ref{thm:spinrefined} is strongest when $s=\beta=0$. This choice yields
\begin{corollary}[Electronic wave functions]\label{cor:product}
  Let $\psi=\psi_\sigma$ be a fixed-spin spatial component of an electronic eigenfunction. For every $0\leq\alpha<1$,
  \begin{equation} \label{eq:productmoment} \int_{\R^{3N}}\left(\sum_{I\in\mathcal I_\sigma}\prod_{i\in I}\langle\xi_i\rangle^\alpha\right)\abs{\widehat\psi(\xi)}\,\dd\xi<\infty. \end{equation}
  If the state is fully spin-polarized, the sum in \eqref{eq:productmoment} consists of the single product over all electron coordinates.
\end{corollary}

The above corollary expresses the Pauli gain as simultaneous product-weighted Fourier integrability over each same-spin block. In a fully polarized state, every electron pair benefits from the antisymmetric difference-kernel estimate.

For $s,\alpha,\beta\geq0$, define
\begin{equation}\label{eq:spin-intersection-space} \mathcal X_\sigma^{s,\alpha,\beta}=\bigcap_{I\in\mathcal I_\sigma}\Bmix{s}{I}{\alpha,\beta}(\R^{3N}), \qquad \norm[\mathcal X_\sigma^{s,\alpha,\beta}]{f}=\sum_{I\in\mathcal I_\sigma}\norm[\Bmix{s}{I}{\alpha,\beta}]{f}. \end{equation}
The following lemma gives the sharp comparison with isotropic spectral Barron spaces. Its proof is deferred to Appendix~\ref{app:comparison-proofs}.

\begin{lemma}[Sharp comparison with isotropic spectral Barron spaces]
  \label{lem:sharp-isotropic-comparison}
  Let $n_\sigma=\max_{I\in\mathcal I_\sigma}\abs I$ and $0\leq\alpha<1$. Then, for every $t\geq0$,
  \begin{equation}\label{eq:sharp-isotropic-comparison}
    \begin{aligned}
      \B^t(\R^{3N})\hookrightarrow\mathcal X_\sigma^{0,\alpha,0}
      &\Longleftrightarrow t\geq\alpha n_\sigma,\\
      \mathcal X_\sigma^{0,\alpha,0}\hookrightarrow\B^t(\R^{3N})
      &\Longleftrightarrow t\leq\alpha.
    \end{aligned}
  \end{equation}
  In particular, \[ \B^{\alpha n_\sigma}(\R^{3N})\hookrightarrow \mathcal X_\sigma^{0,\alpha,0}\hookrightarrow \B^\alpha(\R^{3N}). \]
  If $\alpha>0$ and $n_\sigma>1$, both inclusions are strict. If $\alpha=0$ or $n_\sigma=1$, then $\mathcal X_\sigma^{0,\alpha,0}=\B^\alpha(\R^{3N})$ with equivalent norms.
\end{lemma}

By Corollary~\ref{cor:product} and the second embedding in \eqref{eq:sharp-isotropic-comparison}, $\psi\in\B^\alpha(\R^{3N})$ for every $0\leq\alpha<1$, recovering the sharp isotropic Barron regularity established in \cite[Theorem~4.7]{Yserentant2026} and \cite[Corollary~2.7]{MingYu2025} without a dimension-dependent constant. If $0<\alpha<1$ and an occupied spin block contains at least two electrons, Lemma~\ref{lem:sharp-isotropic-comparison} shows that the mixed product-moment conclusion is strictly stronger than ordinary isotropic Barron regularity.

\subsection{Mixed Fourier--Lebesgue interpolation}\label{subsec:FLinterpolation}

We combine the Fourier $L^1$ estimates with Meng's mixed $L^2$ regularity to obtain an interpolated Fourier--Lebesgue scale. Meng's fixed-spin result gives mixed-Sobolev regularity under $\alpha<5/4$, $\beta<3/4$, and $\alpha+\beta<3/2$ \cite[Corollary~2.4]{Meng2023}. The two estimates encode different integrability and decay properties. Since both hold for the same eigenfunction, interpolation of their weighted Fourier norms yields the following intermediate regularity for every $1<p<2$.

\begin{corollary}[Mixed Fourier--Lebesgue interpolation]\label{cor:FLinterpolation}
  Let $\psi=\psi_\sigma$ be a fixed-spin spatial component of an electronic eigenfunction and fix $I\in\mathcal I_\sigma$. Let $s_0,\alpha_0,\beta_0\geq0$ satisfy the applicable condition in \eqref{eq:spinconditions}, with $s,\alpha,\beta$ replaced by $s_0,\alpha_0,\beta_0$. Let $0\leq\alpha_2<5/4$ and $0\leq\beta_2<3/4$ satisfy $\alpha_2+\beta_2<3/2$. For $1<p<2$, set $\theta=2(1-1/p)$ and \[ s_\theta=(1-\theta)s_0+\theta,\qquad \alpha_\theta=(1-\theta)\alpha_0+\theta\alpha_2,\qquad \beta_\theta=(1-\theta)\beta_0+\theta\beta_2. \]
  Then \[ w_{s_\theta,I}^{\alpha_\theta,\beta_\theta}\widehat\psi\in L^p(\R^{3N}). \]
\end{corollary}

The proof is deferred to the end of Section~\ref{sec:eigenfunction}.

\subsection{Sharpness of the main results}\label{subsec:main-sharpness}

Hydrogen, helium, and lithium ground states prove uniform sharpness in the clamped-nuclei Coulomb class. The proofs are given in Section~\ref{sec:consequences}. Let $\Phi_{N,Z}$ denote the positive scalar ground state of the $N$-electron atom with nuclear charge $Z$. Such a ground state exists for $N<Z+1$ by Zhislin's binding theorem \cite{Zhislin1960}; see also \cite[Corollary~11.10]{Simon2019}.

\begin{proposition}[Atomic endpoint tests]\label{prop:atomicsharpness}
  Let $s,\alpha,\beta\geq0$.
  \begin{enumerate}[label=\textup{(\roman*)}]
    \item For $N=1$ and $I=\{1\}$, the hydrogen ground state $\psi_{\rm H}(x)=e^{-\abs x}$ satisfies \[ \psi_{\rm H}\in\Bmix{s}{I}{\alpha,\beta}(\R^3) \quad\Longleftrightarrow\quad s+\alpha<1. \]
    \item For the helium scalar ground state and $I=\{1\}$, \[ \Phi_{2,2}\in\Bmix{s}{I}{\alpha,\beta}(\R^6) \quad\Longleftrightarrow\quad s+\alpha+\beta<1. \]
    \item For the lithium scalar ground state and $I=\{1\}$, \[ \Phi_{3,3}\in\Bmix{s}{I}{\alpha,\beta}(\R^9) \quad\Longleftrightarrow\quad s+\max\{\alpha+\beta,2\beta\}<1. \]
  \end{enumerate}
\end{proposition}

\begin{corollary}[Uniform sharpness of the index-set family]\label{cor:uniformsharpness}
  The three parameter regions in \eqref{eq:mainconditions} are uniformly sharp in $N$, the nonempty index set $I$, and the clamped-nuclei Coulomb Hamiltonian \eqref{eq:Hamiltonian}.
\end{corollary}

The same examples prove uniform sharpness of the spin-dependent conclusions. Multiplying the helium spatial ground state by the spin singlet gives a physical fermionic state with one electron in each of the two occupied spin blocks, so the case $\abs{\mathcal I_\sigma}=2$ in \eqref{eq:spinconditions} is sharp. Hydrogen makes $s+\alpha<1$ sharp in the fully spin-polarized class when $N$ is allowed to vary. If $q\geq3$, multiplying the lithium spatial ground state by a totally antisymmetric three-spin factor and taking a fixed-spin component with three distinct occupied spin states realizes both the $\alpha+\beta$ and $2\beta$ faces.

\section{Weighted Coulomb operators}\label{sec:operators}

The Fourier transform of $\abs x^{-1}$ in three dimensions is a positive multiple of $\abs\xi^{-2}$. Nuclear terms therefore lead to a scalar Riesz potential, while an electron--electron term translates two momentum variables in opposite directions. We first derive sharp weighted $L^1$ bounds for the scalar and ordinary pair operators, and then exploit antisymmetric cancellation to improve the pair threshold. These estimates provide the convolution bounds used in Section~\ref{sec:gain} to obtain a positive resolvent gain.

\subsection{Scalar and ordinary pair estimates}

This subsection computes the exact weighted $L^1$ norms of the nuclear convolution and the ordinary electron--pair translation operator. Their endpoint failures determine the restrictions used for nuclear and non-antisymmetric pair terms in Section~\ref{sec:gain}. For $f\in L^1(\R^3)$, define
\begin{equation*} (I_2f)(x){:}=\int_{\R^3}\frac{f(y)}{\abs{x-y}^2}\,\dd y. \end{equation*}
\begin{proposition}[Scalar Coulomb convolution]\label{prop:scalar}
  For $\tau<-1$, the exact norm of $I_2\colon L^1(\R^3)\to L^1(\R^3,\langle x\rangle^\tau\dd x)$ is \begin{equation*} 2\pi^{3/2}\frac{\Gamma((-\tau-1)/2)}{\Gamma(-\tau/2)}. \end{equation*}
  The mapping is unbounded when $\tau\geq-1$.
\end{proposition}

\begin{proof}
  Set $w_\tau(x)=\langle x\rangle^\tau$ and \[ \Lambda_\tau(y)=\int_{\R^3}\frac{w_\tau(x)}{\abs{x-y}^2}\,\dd x. \]
  For $f\in L^1(\R^3)$, positivity of the kernel and Tonelli's theorem give \[ \begin{aligned} \norm[L^1(w_\tau\dd x)]{I_2f} &\leq \int_{\R^3}w_\tau(x)\int_{\R^3}\frac{\abs{f(y)}}{\abs{x-y}^2}\,\dd y\,\dd x \\ &=\int_{\R^3}\abs{f(y)}\Lambda_\tau(y)\,\dd y \leq \bigl(\operatorname*{ess\,sup}_{y\in\R^3}\Lambda_\tau(y)\bigr)\norm[L^1]{f}. \end{aligned} \]
  The functions $w_\tau$ and $\abs{\,\cdot\,}^{-2}$ are nonnegative, radial, radially nonincreasing, and vanish at infinity. The Brascamp--Lieb--Luttinger rearrangement inequality \cite[Theorem~3.4, p.~233; see also Definition~3.3, p.~232]{BrascampLiebLuttinger1974} therefore gives, for every $y\in\R^3$, \[ \Lambda_\tau(y)=\int_{\R^3}\frac{w_\tau(x)}{\abs{x-y}^2}\,\dd x \leq \int_{\R^3}\frac{w_\tau(x)}{\abs x^2}\,\dd x=\Lambda_\tau(0)<\infty, \] where finiteness follows from $\tau<-1$. The function $\Lambda_\tau$ is continuous.
  Hence, for every $\varepsilon>0$, there is a ball $B_\varepsilon$ centered at the origin on which $\Lambda_\tau(y)>\Lambda_\tau(0)-\varepsilon$. Taking $f=\abs{B_\varepsilon}^{-1}\mathbf 1_{B_\varepsilon}$ and using Tonelli once more gives \[ \norm[L^1(w_\tau\dd x)]{I_2f}=\frac1{\abs{B_\varepsilon}}\int_{B_\varepsilon} \Lambda_\tau(y)\,\dd y>\Lambda_\tau(0)-\varepsilon. \]
  Together with the upper bound, and then letting $\varepsilon\downarrow0$, this proves \[ \norm{I_2}=\operatorname*{ess\,sup}_{y\in\R^3}\Lambda_\tau(y)=\Lambda_\tau(0). \]
Using spherical coordinates, the substitution $v=r^2$, and Euler's beta integral, we obtain
\[ \Lambda_\tau(0)=\int_{\R^3}\frac{\langle x\rangle^\tau}{\abs x^2}\,\dd x =4\pi\int_0^\infty(1+r^2)^{\tau/2}\,\dd r =2\pi^{3/2}\frac{\Gamma((-\tau-1)/2)}{\Gamma(-\tau/2)}. \]
For $\tau\geq-1$, choose $R>0$ and a nonnegative $f\in L^1(\R^3)$ supported in $B_R$ with $\int_{\R^3}f(y)\,\dd y=1$. If $\abs x\geq2R$, then $\abs{x-y}\leq\abs x+R\leq3\abs x/2$ for $y\in\operatorname{supp}f$, and therefore \[ (I_2f)(x)\geq\frac{4}{9\abs x^2}. \]
Consequently, \[ \norm[L^1(w_\tau\dd x)]{I_2f}\geq\frac49\int_{\abs x\geq2R}\frac{\langle x\rangle^\tau}{\abs x^2}\,\dd x =\frac{16\pi}{9}\int_{2R}^\infty(1+r^2)^{\tau/2}\,\dd r=\infty \] whenever $\tau\geq-1$, which proves the asserted unboundedness.
\end{proof}

For integrable functions defined on $\R^3\times\R^3$, define the pair translation operator
\begin{equation*} (Tf)(x,y){:}=\int_{\R^3}\frac{f(x-k,y+k)}{\abs k^2}\,\dd k. \end{equation*}
\begin{proposition}[Ordinary pair operator]\label{prop:ordinary}
  For $t<-1$, the exact norm of \[ T\colon L^1(\R^6)\longrightarrow L^1(\R^6,\langle(x,y)\rangle^t\dd x\dd y) \] is $2^{-1/2}$ times the scalar norm in Proposition~\ref{prop:scalar}, namely \begin{equation*} C_t^{\mathrm o}=\sqrt2\,\pi^{3/2} \frac{\Gamma((-t-1)/2)}{\Gamma(-t/2)}. \end{equation*}
  The mapping is unbounded when $t\geq-1$.
\end{proposition}

\begin{proof}
  For $t\in\R$ and $(u,v)\in\R^6$, set \[ Q_t(u,v)=\int_{\R^3}\langle(u+k,v-k)\rangle^t\frac{\dd k}{\abs k^2}\in(0,\infty]. \]
  Suppose first that $t<-1$.
  Tonelli's theorem gives \[ \norm[L^1(\langle(x,y)\rangle^t\dd x\dd y)]{Tf} \leq\int_{\R^6}\abs{f(u,v)}Q_t(u,v)\,\dd u\,\dd v. \]
  Equality holds here for every nonnegative $f$. Testing with normalized characteristic functions of finite-measure subsets of \[ \{(u,v)\mid Q_t(u,v)>\operatorname*{ess\,sup}Q_t-\varepsilon\} \] and then letting $\varepsilon\downarrow0$ proves \[ \norm{T}=\operatorname*{ess\,sup}_{(u,v)\in\R^6}Q_t(u,v). \]

  Put $\lambda=1+\frac12\abs{u+v}^2$ and $b=(u-v)/2$.
  Since \[ \abs{u+k}^2+\abs{v-k}^2=2\abs{k+b}^2+\frac12\abs{u+v}^2, \] the change of variables $q=\sqrt{2/\lambda}(k+b)$ gives \[ Q_t(u,v)=\frac{\lambda^{(t+1)/2}}{\sqrt2}\,\Lambda_t\!\left(\sqrt{\frac2\lambda}\,b\right), \] where $\Lambda_t$ is the scalar convolution weight used in the proof of Proposition~\ref{prop:scalar}. That proof shows that $\Lambda_t$ is continuous and $\Lambda_t(z)\leq\Lambda_t(0)$ for every $z\in\R^3$. Since $t<-1$ and $\lambda\geq1$, \[ Q_t(u,v)\leq\frac1{\sqrt2}\Lambda_t(0)=Q_t(0,0). \]
  The displayed representation makes $Q_t$ continuous, so its pointwise maximum is also its essential supremum. Proposition~\ref{prop:scalar} now gives \[ \norm{T}=\frac1{\sqrt2}\Lambda_t(0)=\sqrt2\,\pi^{3/2}\frac{\Gamma((-t-1)/2)}{\Gamma(-t/2)}=C_t^{\mathrm o}. \]

  Suppose $t\geq-1$. For each fixed $(u,v)$, there are $R>0$ and $c>0$ such that \[ \langle(u+k,v-k)\rangle^t\geq c\abs k^t, \qquad \abs k\geq R. \]
  Thus \[ Q_t(u,v)\geq4\pi c\int_R^\infty r^t\,\dd r=\infty. \]
  Applying Tonelli's theorem to any nonzero nonnegative $f\in L^1(\R^6)$ proves that $T$ is unbounded.
\end{proof}

\subsection{The antisymmetric gain}

Compared with Proposition~\ref{prop:ordinary}, the next result extends the admissible target exponent from $t<-1$ to $t<0$, thereby gaining one full order from antisymmetry. We first derive the weighted $L^1$ estimate from the difference-kernel representation of the pair operator and then record its fiberwise form for use in Section~\ref{sec:gain}.

\begin{proposition}[Antisymmetric pair operator]\label{prop:anti}
  Let $s,a\geq0$ and $-1\leq t<0$. If \begin{equation*} s+a\geq t+1, \end{equation*} then, on functions satisfying $f(x,y)=-f(y,x)$, \begin{equation} \label{eq:antibound} \int_{\R^6}\langle(x,y)\rangle^t\abs{Tf(x,y)}\,\dd x\dd y \leq C_t^{\mathrm a} \int_{\R^6}\langle(x,y)\rangle^s (\langle x\rangle\langle y\rangle)^a \abs{f(x,y)}\,\dd x\dd y, \end{equation} where \begin{equation*} C_t^{\mathrm a} =-\frac{2^{t/2+1}\pi^2}{t+2}\cot\!\left(\frac{\pi t}{4}\right). \end{equation*}
\end{proposition}

\begin{proof}
  It suffices first to consider antisymmetric $f\in C_c^\infty(\R^6)$. The space $C_c^\infty(\R^6)$ is dense in the weighted input space, and the antisymmetrization \[ (P_-g)(x,y)=\frac12\bigl(g(x,y)-g(y,x)\bigr) \] is contractive because the input weight is invariant under $x\leftrightarrow y$. The estimate below therefore extends to every antisymmetric input by density.

  Put $d_{xy}=x-y$. Substituting $k\mapsto d_{xy}-k$ and using antisymmetry give
  \[
  \begin{aligned}
    Tf(x,y)
    &=-\int_{\R^3}\frac{f(x-k,y+k)}{\abs{k-d_{xy}}^2}\,\dd k\\
    &=\frac12\int_{\R^3}f(x-k,y+k)
      \left(\frac1{\abs k^2}-\frac1{\abs{k-d_{xy}}^2}\right)\dd k.
  \end{aligned}
  \]
  Set $u=x-k$, $v=y+k$, $X=u-v$, $Y=u+v$, and $z=k+X/2$.
  Then \[ x=z+\frac Y2,\qquad y=\frac Y2-z,\qquad \left\langle z+\frac Y2,\frac Y2-z\right\rangle^2 =1+2\abs z^2+\frac12\abs Y^2. \]
  Since $t<0$, \[ \left\langle z+\frac Y2,\frac Y2-z\right\rangle^t\leq(1+2\abs z^2)^{t/2}. \] Using this inequality and then applying Tonelli's theorem to the nonnegative integrand gives \[ \int_{\R^6}\langle(x,y)\rangle^t\abs{Tf(x,y)}\,\dd x\,\dd y \leq\frac12\int_{\R^6}\abs{f(u,v)} J_t\!\left(\frac{u-v}{2}\right)\dd u\,\dd v, \] where \[ J_t(W)=\int_{\R^3}(1+2\abs z^2)^{t/2} \left|\frac1{\abs{z-W}^2}-\frac1{\abs{z+W}^2}\right|\dd z. \]

  Suppose $W\neq0$. With $e=W/\abs W$, substitute $z=\abs W r\omega$, where $r>0$ and $\omega\in\mathbb S^2$. Since $t<0$,
  \[ \begin{aligned} J_t(W) &\leq2^{t/2}\abs W^{t+1} \int_0^\infty r^{t+2}\int_{\mathbb S^2} \left|\frac1{\abs{r\omega-e}^2} -\frac1{\abs{r\omega+e}^2}\right|\dd\omega\,\dd r. \end{aligned} \]
  Direct integration in the polar variable $\omega\cdot e$ yields \[ \int_{\mathbb S^2} \left|\frac1{\abs{r\omega-e}^2} -\frac1{\abs{r\omega+e}^2}\right|\dd\omega =\frac{4\pi}{r}\log\left|\frac{r^2+1}{r^2-1}\right|. \]
  Hence \[ J_t(W)\leq2^{t/2+2}\pi \Theta_t\abs W^{t+1}, \] with \[ \Theta_t=\int_0^\infty r^{t+1} \log\left|\frac{r^2+1}{r^2-1}\right|\dd r. \]
  For $-1\leq t<0$, the substitution $\varrho=r^2$, inversion on $(1,\infty)$, and the nonnegative expansion \[ \log\frac{1+\varrho}{1-\varrho}=2\sum_{m=0}^\infty\frac{\varrho^{2m+1}}{2m+1}, \qquad 0<\varrho<1, \] give
  \[
  \begin{aligned}
    \Theta_t
    &=\sum_{m=0}^\infty\frac1{2m+1}
      \left(\frac1{2m+t/2+2}+\frac1{2m-t/2}\right)\\
    &=\frac1{t+2}\sum_{m=0}^\infty
      \left(\frac1{m-t/4}-\frac1{m+1+t/4}\right)\\
    &=-\frac\pi{t+2}\cot\!\left(\frac{\pi t}{4}\right).
  \end{aligned}
  \]
  The last equality follows from \[ \sum_{m=0}^\infty \left(\frac1{m+\theta}-\frac1{m+1-\theta}\right) =\pi\cot(\pi\theta), \qquad 0<\theta<1, \] evaluated at $\theta=-t/4$. Thus \[ 2^{t/2+2}\pi \Theta_t=-\frac{2^{t/2+2}\pi^2}{t+2} \cot\!\left(\frac{\pi t}{4}\right)=2C_t^{\mathrm a}. \]
  If $W=0$, then $J_t(W)=0$.

  Since \[ \langle(u,v)\rangle^2=1+2\abs W^2+\frac12\abs Y^2, \qquad \langle u\rangle^2\langle v\rangle^2\geq\langle(u,v)\rangle^2, \] we have, for $W\neq0$, \[ \abs W^{t+1}\leq(1+2\abs W^2)^{(t+1)/2}\leq(1+2\abs W^2)^{(s+a)/2}\leq\langle(u,v)\rangle^s(\langle u\rangle\langle v\rangle)^a. \]
  Together with $J_t(0)=0$, this yields \[ J_t(W)\leq2C_t^{\mathrm a}\langle(u,v)\rangle^s (\langle u\rangle\langle v\rangle)^a. \]
  Substitution into the preceding integral estimate gives \eqref{eq:antibound}. The density argument completes the proof.
\end{proof}

The estimates lift directly to additional spectator variables. We record the form used later.

\begin{lemma}[Fiberwise lifting]\label{lem:fiber}
  Let $S$ act only on $z\in\R^n$ and suppose that, for every $g$ in its domain, \[ \int_{\R^n}W_{\rm out}(z)\abs{Sg(z)}\,\dd z \leq C\int_{\R^n}W_{\rm in}(z)\abs{g(z)}\,\dd z. \]
  Let $F$ be measurable, suppose that $F_\zeta(z)=F(z,\zeta)$ belongs to this domain for almost every $\zeta\in\R^m$, and assume that $(z,\zeta)\mapsto SF_\zeta(z)$ is measurable. If $b\colon\R^m\to[0,\infty]$ is measurable and the right-hand side below is finite, then
  \[
  \begin{aligned}
    &\int_{\R^m}b(\zeta)\int_{\R^n}
      W_{\rm out}(z)\abs{S F_\zeta(z)}\,\dd z\,\dd\zeta\\
    &\qquad\leq C\int_{\R^m}b(\zeta)\int_{\R^n}
      W_{\rm in}(z)\abs{F_\zeta(z)}\,\dd z\,\dd\zeta.
  \end{aligned}
  \]
  This applies to Propositions~\ref{prop:scalar}--\ref{prop:anti}. In the antisymmetric case, assume that $F_\zeta(x,y)$ is antisymmetric in $(x,y)$ for almost every $\zeta$. The joint weights may then be used because \[ \langle(x,y,\zeta)\rangle^t\leq\langle(x,y)\rangle^t, \qquad \langle(x,y,\zeta)\rangle^s\geq\langle(x,y)\rangle^s \] for $s\geq0$ and $t<0$.
\end{lemma}

\begin{proof}
  Apply the active-variable estimate to $F_\zeta$ for almost every $\zeta$, multiply it by $b(\zeta)$, and integrate in $\zeta$. Tonelli's theorem justifies the order of integration and yields the stated estimate.
\end{proof}

\section{A mixed-regularity gain for the Coulomb resolvent}\label{sec:gain}

This section combines the convolution bounds of Section~\ref{sec:operators} with the two-order decay of the free resolvent. We first establish density and the basic mixed Barron gain, then identify the Fourier-side extension with the $H^1$ realization and incorporate all same-spin antisymmetries. These mapping results supply the high-frequency contraction used in Section~\ref{sec:eigenfunction}.

Fix $\mu>0$ and set \[ \mathcal R_\mu=\left(-\frac12\Delta+\mu\right)^{-1}. \]
Its Fourier multiplier $r_\mu(\xi)=(\abs\xi^2/2+\mu)^{-1}$ satisfies
\begin{equation} \label{eq:resolventbound} r_\mu(\xi)\leq\max\{2,\mu^{-1}\}\langle\xi\rangle^{-2}. \end{equation}

\begin{lemma}[Density]\label{lem:density}
  Let $I\subseteq\{1,\dots,N\}$ be nonempty and let $s,\alpha,\beta\geq0$. The Schwartz functions antisymmetric with respect to $I$ are dense in the antisymmetric subspace of $\Bmix{s}{I}{\alpha,\beta}$. They are also dense in the antisymmetric subspace of
  \begin{equation} \label{eq:intersectiondensity} H^1(\R^{3N})\cap\Bmix{s}{I}{\alpha,\beta}, \end{equation}
  equipped with the sum norm.
\end{lemma}

\begin{proof}
  Write $w=w_{s,I}^{\alpha,\beta}$ and let $u\in\Bmix{s}{I}{\alpha,\beta}$ be antisymmetric with respect to $I$, assuming also that $u\in H^1$ for the second assertion. Truncation on expanding balls, followed by mollification and a diagonal choice, gives $h_n\in C_c^\infty(\R^{3N})$ such that $h_n\to\widehat u$ in $L^1(w\dd\xi)$ and, when $u\in H^1$, in $L^2(\langle\xi\rangle^2\dd\xi)$. Here truncation converges by dominated convergence, while on each fixed enlarged ball the weights are bounded and the standard $L^1$ and $L^2$ approximate-identity theorem applies \cite[Theorem~8.14]{Folland1999}.

  Let $\mathfrak S_I$ be the group of permutations of the indices in $I$, extended by the identity on $I^c$, and set \[ (\mathcal P_Ig)(\xi)=\frac1{\abs{\mathfrak S_I}}\sum_{\pi\in\mathfrak S_I}\operatorname{sgn}(\pi)g(\pi^{-1}\xi). \]
  Every coordinate permutation $\pi$ satisfying $\pi(I)=I$ obeys
  \begin{equation}\label{eq:weight-permutation-invariance} w_{s,I}^{\alpha,\beta}(\pi\xi)=w_{s,I}^{\alpha,\beta}(\xi), \qquad \langle\pi\xi\rangle=\langle\xi\rangle. \end{equation}
  By \eqref{eq:weight-permutation-invariance} and the change of variables $\eta=\pi^{-1}\xi$, every such coordinate permutation is an isometry in both weighted spaces. Hence, by the triangle inequality, \[ \|\mathcal P_I g\|_{L^1(w\,\dd\xi)}\leq \|g\|_{L^1(w\,\dd\xi)}, \qquad \|\mathcal P_I g\|_{L^2(\langle\xi\rangle^2\,\dd\xi)}\leq \|g\|_{L^2(\langle\xi\rangle^2\,\dd\xi)}. \]
  Since $\mathcal P_I\widehat u=\widehat u$, the functions defined by $\widehat u_n=\mathcal P_Ih_n$ belong to $C_c^\infty(\R^{3N})$, are antisymmetric with respect to $I$, and converge to $\widehat u$ in the required norm or sum norm. Their inverse Fourier transforms belong to $\Sch(\R^{3N})$, which proves both assertions.
\end{proof}

For the operator estimates, fix a nonempty set $I\subseteq\{1,\dots,N\}$, let $s,\alpha,\beta\geq0$, and write
\begin{equation}\label{eq:ai}
  a_i=\begin{cases}\alpha,&i\in I,\\ \beta,&i\notin I,\end{cases} \qquad A_I=\max_{1\leq i\leq N}a_i.
\end{equation}
Define \[ D_I=\max\Bigl(\{a_i+a_j\mid1\leq i<j\leq N,\ \{i,j\}\nsubseteq I\}\cup\{0\}\Bigr). \]
Then
\begin{equation}\label{eq:DI-values}
  D_I=\begin{cases}
    0,&I^c=\varnothing,\\
    \alpha+\beta,&\abs{I^c}=1,\\
    \max\{\alpha+\beta,2\beta\},&\abs{I^c}\geq2.
  \end{cases}
\end{equation}
Moreover, $A_I=\alpha$ if $I^c=\varnothing$ and $A_I=\max\{\alpha,\beta\}$ otherwise. In the latter case $D_I\geq A_I$, while in the former $D_I=0$. Hence \eqref{eq:mainconditions} is equivalent to
\begin{equation}\label{eq:operator-mainconditions} s+A_I<1,\qquad s+D_I<1. \end{equation}

\begin{theorem}[Coulomb--resolvent gain]\label{thm:gain}
  Let $I\neq\varnothing$, and let $s,\alpha,\beta\geq0$ satisfy \eqref{eq:mainconditions}. The operator $\mathcal R_\mu V$, initially defined on Schwartz functions that are antisymmetric with respect to $I$, extends uniquely to a bounded map
  \begin{equation} \label{eq:gainmap} \mathcal R_\mu V\colon\Bmix{s}{I}{\alpha,\beta}\cap \{u\mid u\text{ is antisymmetric with respect to }I\} \longrightarrow\Bmix{s+\delta}{I}{\alpha,\beta} \end{equation}
  for some $\delta>0$. One may take any positive $\delta$ satisfying
  \begin{equation} \label{eq:deltachoice} \delta<1-s-A_I,\qquad \delta<1-s-D_I. \end{equation}
\end{theorem}

We prove that the composition $\mathcal R_\mu V$ gains a positive isotropic order $\delta$ in the mixed spectral Barron scale.

\begin{proof}
  By \eqref{eq:operator-mainconditions}, both upper bounds in \eqref{eq:deltachoice} are positive, so such a $\delta$ exists. Since $I\neq\varnothing$, one has $\alpha\leq A_I$, and therefore \[ \delta<1-s-A_I\leq1-\alpha,\qquad \delta<1-s-A_I<2-s-2\alpha. \]
  Here the final strict inequality follows because $(1-\alpha)+(A_I-\alpha)>0$. Thus the two bounds required below for pairs contained in $I$ follow automatically from \eqref{eq:deltachoice}.

  Take $u\in\Sch(\R^{3N})$ with the prescribed antisymmetry. Under the unitary Fourier convention, set $\kappa_{\rm C}=(2\pi^2)^{-1}$. The Coulomb terms satisfy
  \[
    \begin{aligned}
      \widehat{\bigl(\abs{x_i-R_\nu}^{-1}u\bigr)}(\xi)
      &=\kappa_{\rm C}\int_{\R^3}\frac{e^{-ik\cdot R_\nu}}{\abs k^2}\widehat u(\xi_1,\ldots,\xi_i-k,\ldots,\xi_N)\,\dd k,\\
      \widehat{\bigl(\abs{x_i-x_j}^{-1}u\bigr)}(\xi)
      &=\kappa_{\rm C}\int_{\R^3}\frac{\widehat u(\xi_1,\ldots,\xi_i-k,\ldots,\xi_j+k,\ldots,\xi_N)}{\abs k^2}\,\dd k.
    \end{aligned}
  \]
  The Coulomb coefficients and their signs are absorbed into the final operator norm. Define \[ W_{\rm in}(\xi)=\langle\xi\rangle^s\prod_{\ell=1}^N\langle\xi_\ell\rangle^{a_\ell}, \qquad W_{\rm out}(\xi)=\langle\xi\rangle^{s+\delta}\prod_{\ell=1}^N\langle\xi_\ell\rangle^{a_\ell}. \]
  Equation~\eqref{eq:resolventbound} gives \[ W_{\rm out}(\xi)r_\mu(\xi) \leq\max\{2,\mu^{-1}\}\langle\xi\rangle^{s+\delta-2} \prod_{\ell=1}^N\langle\xi_\ell\rangle^{a_\ell}. \]

  Consider a nuclear term acting on the $i$th coordinate. Write \[ \zeta=(\xi_\ell)_{\ell\neq i},\qquad b_i(\zeta)=\prod_{\ell\neq i}\langle\xi_\ell\rangle^{a_\ell}, \qquad \tau_i=s+\delta-2+a_i. \]
  Since $a_i\leq A_I$, \eqref{eq:deltachoice} gives \[ \tau_i\leq s+\delta-2+A_I<-1. \]
  Since $\langle\xi_i\rangle^{a_i}\leq\langle\xi\rangle^{a_i}$ and $\tau_i<0$, \[ W_{\rm out}(\xi)r_\mu(\xi)\leq Cb_i(\zeta)\langle\xi\rangle^{\tau_i}\leq Cb_i(\zeta)\langle\xi_i\rangle^{\tau_i}. \]
  Proposition~\ref{prop:scalar}, applied for almost every fixed $\zeta$, yields
  \begin{align}
    &\norm[\Bmix{s+\delta}{I}{\alpha,\beta}]
      {\mathcal R_\mu(\abs{x_i-R_\nu}^{-1}u)}\notag\\
    &\quad\leq C\int_{\R^{3(N-1)}}b_i(\zeta)
      \int_{\R^3}\langle\xi_i\rangle^{\tau_i}
      \int_{\R^3}\frac{\abs{\widehat u(\xi_i-k,\zeta)}}{\abs k^2}
      \,\dd k\,\dd\xi_i\,\dd\zeta\notag\\
    &\quad\leq C\int_{\R^{3N}}b_i(\zeta)\abs{\widehat u(\xi)}\,\dd\xi
      \leq C\norm[\Bmix{s}{I}{\alpha,\beta}]{u}.
    \label{eq:nucleartermbound}
  \end{align}
  The last inequality follows from $b_i(\zeta)\leq W_{\rm in}(\xi)$.

  Let $i<j$ and write \[ \zeta=(\xi_\ell)_{\ell\notin\{i,j\}}, \qquad b_{ij}(\zeta)=\prod_{\ell\notin\{i,j\}}\langle\xi_\ell\rangle^{a_\ell}. \]
  For functions of all momentum variables, let \[ (T_{ij}f)(\xi)=\int_{\R^3} \frac{f(\xi_1,\ldots,\xi_i-k,\ldots,\xi_j+k,\ldots,\xi_N)}{\abs k^2}\,\dd k. \]
  For $\{i,j\}\nsubseteq I$, set $t_{ij}=s+\delta-2+a_i+a_j$.
  The definition of $D_I$ and \eqref{eq:deltachoice} imply \[ t_{ij}\leq s+\delta-2+D_I<-1. \]
  Since $t_{ij}<0$, \[ W_{\rm out}(\xi)r_\mu(\xi) \leq C b_{ij}(\zeta)\langle\xi\rangle^{t_{ij}} \leq C b_{ij}(\zeta)\langle(\xi_i,\xi_j)\rangle^{t_{ij}}. \]
  Proposition~\ref{prop:ordinary}, applied fiberwise, and the inequality $b_{ij}(\zeta)\leq W_{\rm in}(\xi)$ give
  \begin{equation} \label{eq:ordinarypairbound} \norm[\Bmix{s+\delta}{I}{\alpha,\beta}] {\mathcal R_\mu(\abs{x_i-x_j}^{-1}u)} \leq C\norm[\Bmix{s}{I}{\alpha,\beta}]{u}. \end{equation}

  Suppose next that $i,j\in I$. Then $a_i=a_j=\alpha$, and $\widehat u$ is antisymmetric under interchange of $\xi_i$ and $\xi_j$. Put \[ t_{ij}=s+\delta-2+2\alpha. \]
  Since $\delta<2-s-2\alpha$, one has $t_{ij}<0$. If $t_{ij}<-1$, Proposition~\ref{prop:ordinary} applies. If $-1\leq t_{ij}<0$, then \[ (s+\alpha)-(t_{ij}+1)=1-\alpha-\delta>0. \]
  Proposition~\ref{prop:anti} and Lemma~\ref{lem:fiber} therefore imply
  \[
  \begin{aligned}
    &\int_{\R^{3(N-2)}}b_{ij}(\zeta)\int_{\R^6}
      \langle(\xi_i,\xi_j,\zeta)\rangle^{t_{ij}}
      \abs{T_{ij}\widehat u(\xi_i,\xi_j,\zeta)}
      \,\dd\xi_i\,\dd\xi_j\,\dd\zeta\\
    &\quad\leq C_{t_{ij}}^{\rm a}\int_{\R^{3N}}b_{ij}(\zeta)
      \langle\xi\rangle^s
      \bigl(\langle\xi_i\rangle\langle\xi_j\rangle\bigr)^\alpha
      \abs{\widehat u(\xi)}\,\dd\xi\\
    &\quad=C_{t_{ij}}^{\rm a}\norm[\Bmix{s}{I}{\alpha,\beta}]{u}.
  \end{aligned}
  \]
  Thus \eqref{eq:ordinarypairbound} also holds when $i,j\in I$.

  Summing \eqref{eq:nucleartermbound} and \eqref{eq:ordinarypairbound} over the finitely many Coulomb terms gives \[ \norm[\Bmix{s+\delta}{I}{\alpha,\beta}]{\mathcal R_\mu Vu} \leq C\norm[\Bmix{s}{I}{\alpha,\beta}]{u}, \] where $C$ is independent of $u$. Lemma~\ref{lem:density} extends this estimate uniquely to the domain in \eqref{eq:gainmap}.
\end{proof}

\begin{lemma}[Sobolev bounds for $V$ and $\mathcal R_\mu$]\label{lem:sobolevbounds}
  There is a constant $C_V>0$ such that
  \begin{equation}\label{eq:V-H1-L2} \norm[L^2]{Vf}\leq C_V\norm[H^1]{f}, \qquad f\in H^1(\R^{3N}). \end{equation}
  Moreover,
  \begin{equation}\label{eq:resolvent-L2-H2} \norm[H^2]{\mathcal R_\mu g}\leq\max\{2,\mu^{-1}\}\norm[L^2]{g}, \qquad g\in L^2(\R^{3N}). \end{equation}
\end{lemma}

\begin{proof}
  Applying the three-dimensional Hardy inequality fiberwise in each collision variable $x_i-R_\nu$ and, after a linear change of variables, in $x_i-x_j$, and then summing the finitely many Coulomb terms gives \eqref{eq:V-H1-L2}; see also \cite[proof of Lemma~4.1]{Yserentant2026}. The multiplier bound \eqref{eq:resolventbound} and Plancherel's theorem give \eqref{eq:resolvent-L2-H2}, which completes the proof.
\end{proof}

\begin{lemma}[Consistency of the extension]\label{lem:consistency}
  The bounded Fourier-side extension of $\mathcal R_\mu V$ furnished by Theorem~\ref{thm:gain} agrees on the antisymmetric subspace of \eqref{eq:intersectiondensity} with the operator obtained from multiplication by $V$ in position space followed by the $L^2$ resolvent.
\end{lemma}

\begin{proof}
  Let $\mathcal E$ denote the Fourier-side extension in Theorem~\ref{thm:gain}. Let $u\in H^1(\R^{3N})\cap\Bmix{s}{I}{\alpha,\beta}$ be antisymmetric with respect to $I$, and let $(u_n)\subset\Sch(\R^{3N})$ be the approximating sequence furnished by Lemma~\ref{lem:density}. Equations~\eqref{eq:V-H1-L2} and~\eqref{eq:resolvent-L2-H2} give \[ \norm[H^2]{\mathcal R_\mu V(u_n-u)} \leq\max\{2,\mu^{-1}\}C_V\norm[H^1]{u_n-u} \longrightarrow0. \]
  Theorem~\ref{thm:gain} also gives \[ \norm[\Bmix{s+\delta}{I}{\alpha,\beta}]{\mathcal E(u_n-u)} \leq C\norm[\Bmix{s}{I}{\alpha,\beta}]{u_n-u} \longrightarrow0. \]
  Since $\mathcal Eu_n=\mathcal R_\mu Vu_n$ for every $n$, their limits in $\Sch'(\R^{3N})$ coincide. Hence $\mathcal Eu=\mathcal R_\mu Vu$, as asserted.
\end{proof}

Set \[ \mathfrak A_\sigma=\bigcap_{J\in\mathcal I_\sigma}\ \bigcap_{\substack{i,j\in J\\i<j}}\ker(U_{ij}+\operatorname{Id})\subset\Sch'(\R^{3N}). \]

For $I\in\mathcal I_\sigma$, use the exponents $a_i$ from \eqref{eq:ai} and define \[ D_{I,\sigma}=\max\Bigl(\{a_i+a_j\mid1\leq i<j\leq N,\ \sigma_i\neq\sigma_j\}\cup\{0\}\Bigr). \]
Since $I$ is an occupied spin block,
\[
  A_I=\begin{cases}
    \alpha,&\abs{\mathcal I_\sigma}=1,\\
    \max\{\alpha,\beta\},&\abs{\mathcal I_\sigma}\geq2,
  \end{cases}
  \qquad
  D_{I,\sigma}=\begin{cases}
    0,&\abs{\mathcal I_\sigma}=1,\\
    \alpha+\beta,&\abs{\mathcal I_\sigma}=2,\\
    \max\{\alpha+\beta,2\beta\},&\abs{\mathcal I_\sigma}\geq3.
  \end{cases}
\]
For $\abs{\mathcal I_\sigma}\geq2$, one has $D_{I,\sigma}\geq A_I$. Thus \eqref{eq:spinconditions} is equivalent, for every $I\in\mathcal I_\sigma$, to
\begin{equation}\label{eq:operator-spinconditions} s+A_I<1,\qquad s+D_{I,\sigma}<1. \end{equation}

\begin{theorem}[Spin-partition Coulomb--resolvent gain]
  \label{thm:spingain} Fix $I\in\mathcal I_\sigma$, let $s,\alpha,\beta\geq0$, and assume \eqref{eq:spinconditions}. The operator $\mathcal R_\mu V$, initially defined on $\Sch(\R^{3N})\cap\mathfrak A_\sigma$, extends uniquely to a bounded map \[ \mathcal R_\mu V\colon \Bmix{s}{I}{\alpha,\beta}\cap\mathfrak A_\sigma \longrightarrow \Bmix{s+\delta}{I}{\alpha,\beta} \] for some $\delta>0$. One may take any positive $\delta$ satisfying
  \begin{align*} \delta&<1-s-A_I,& \delta&<1-s-D_{I,\sigma}. \end{align*}
\end{theorem}

\begin{proof}
  Equation~\eqref{eq:operator-spinconditions} makes both upper bounds imposed on $\delta$ positive.

  For a nuclear term, the argument leading to \eqref{eq:nucleartermbound} uses only $\delta<1-s-A_I$ and therefore applies unchanged.

  Suppose that $i$ and $j$ belong to different spin blocks. Then \[ a_i+a_j\leq D_{I,\sigma}, \qquad t_{ij}=s+\delta-2+a_i+a_j \leq s+\delta-2+D_{I,\sigma}<-1. \]
  Thus \eqref{eq:ordinarypairbound} holds for every pair in distinct spin blocks.

  Suppose instead that $i$ and $j$ lie in the same spin block. Their coordinate exponents have a common value $c\in\{\alpha,\beta\}$, and $\widehat u$ is antisymmetric in $(\xi_i,\xi_j)$. Since $c\leq A_I$ and $s+c\leq s+A_I<1$, \[ \delta<1-s-A_I\leq1-c<2-s-2c. \]
  With $a_i=a_j=c$, the antisymmetric-pair argument in the proof of Theorem~\ref{thm:gain} gives $t_{ij}=s+\delta-2+2c<0$ and $s+c-(t_{ij}+1)=1-c-\delta>0$, so \eqref{eq:ordinarypairbound} holds. These two cases cover every electron pair, and summation proves the asserted boundedness.

  For the density argument, define \[ \mathcal P_\sigma= \prod_{J\in\mathcal I_\sigma}\left( \frac1{\abs{\mathfrak S_J}} \sum_{\pi\in\mathfrak S_J}\operatorname{sgn}(\pi)U_\pi \right), \] where $U_\pi$ permutes the variables in $J$ and $\widehat{U_\pi f}=U_\pi\widehat f$. Each spin block lies entirely in $I$ or in $I^c$, so every $\pi\in\mathfrak S_J$ satisfies $\pi(I)=I$ and \eqref{eq:weight-permutation-invariance} applies. Hence \[ \norm[\Bmix{s}{I}{\alpha,\beta}]{\mathcal P_\sigma f} \leq\norm[\Bmix{s}{I}{\alpha,\beta}]{f}, \qquad \norm[H^1]{\mathcal P_\sigma f}\leq\norm[H^1]{f}. \]
  Given $u\in\Bmix{s}{I}{\alpha,\beta}\cap\mathfrak A_\sigma$, take the sequence $(u_n)\subset\Sch(\R^{3N})$ furnished by Lemma~\ref{lem:density}. Then $v_n=\mathcal P_\sigma u_n$ satisfies $v_n\to\mathcal P_\sigma u=u$ in the mixed Barron norm, and also in $H^1$ when $u\in H^1$. This proves density in $\mathfrak A_\sigma$, and the argument of Lemma~\ref{lem:consistency} proves consistency of the extension.
\end{proof}

\begin{remark}\label{rem:regions}
  The last case in \eqref{eq:DI-values} is the ordinary-pair restriction obtained when only antisymmetry on $I$ is used. When additional spin-block antisymmetries or some pair types are absent, Theorem~\ref{thm:spinrefined} retains the sharper value $D_{I,\sigma}$.
\end{remark}

\section{Proof of the eigenfunction theorem}\label{sec:eigenfunction}

The gain in Theorem~\ref{thm:gain} makes the high-frequency part of the resolvent fixed-point map contractive, while Plancherel's theorem controls the low-frequency part. We formulate the eigenvalue equation as a fixed point, prove contraction on the relevant $H^1$ and mixed Barron spaces, and identify the solution by a Neumann series. This adapts the argument of \cite[Lemmas~4.5--4.6 and Theorem~4.7]{Yserentant2026} to the mixed spectral Barron spaces and proves Theorems~\ref{thm:mainI} and~\ref{thm:spinrefined}.

For a nonempty set $I\subseteq\{1,\dots,N\}$, set \[ \mathfrak A_I=\{u\in\Sch'(\R^{3N})\mid U_{ij}u=-u\text{ for all distinct }i,j\in I\}, \] and choose $\delta>0$ as in Theorem~\ref{thm:gain}.

Fix $\mu>0$. The estimate~\eqref{eq:V-H1-L2} gives $V\psi\in L^2(\R^{3N})$. The eigenvalue equation therefore implies \[ \left(-\frac12\Delta+\mu\right)\psi=\bigl((E+\mu)\operatorname{Id}-V\bigr)\psi \] in $L^2(\R^{3N})$. Applying $\mathcal R_\mu$ yields
\begin{equation} \label{eq:fixedpoint} \psi=\mathcal A\psi, \qquad \mathcal A=\mathcal R_\mu\bigl((E+\mu)\operatorname{Id}-V\bigr). \end{equation}

The choice of $\delta$ implies $\delta<1-s-A_I\leq1$. Since $w_{s+\delta,I}^{\alpha,\beta}(\xi)=\langle\xi\rangle^\delta w_{s,I}^{\alpha,\beta}(\xi)$ and $\sup_{\xi\in\R^{3N}}\langle\xi\rangle^\delta/(\abs\xi^2/2+\mu)<\infty$, the scalar part of $\mathcal A$ satisfies \[ \norm[\Bmix{s+\delta}{I}{\alpha,\beta}]{\mathcal R_\mu(E+\mu)u} \leq C_{\mu,\delta}\abs{E+\mu} \norm[\Bmix{s}{I}{\alpha,\beta}]{u},
\]
which together with Theorem~\ref{thm:gain} gives
\begin{equation} \label{eq:Abarron} \norm[\Bmix{s+\delta}{I}{\alpha,\beta}]{\mathcal Au} \leq C_{\rm B}\norm[\Bmix{s}{I}{\alpha,\beta}]{u}, \qquad u\in\Bmix{s}{I}{\alpha,\beta}\cap\mathfrak A_I.
\end{equation}
The estimates~\eqref{eq:V-H1-L2} and~\eqref{eq:resolvent-L2-H2} give \[ \norm[H^2]{\mathcal Au}\leq\max\{2,\mu^{-1}\}\bigl(\abs{E+\mu}\norm[L^2]{u}+\norm[L^2]{Vu}\bigr)\leq C_{\rm H}\norm[H^1]{u}. \]
Lemma~\ref{lem:consistency} also ensures that this position-space realization of $\mathcal A$ agrees on \[ H^1(\R^{3N})\cap\Bmix{s}{I}{\alpha,\beta}(\R^{3N})\cap\mathfrak A_I \] with the Fourier-side extension used in \eqref{eq:Abarron}.

For $K>0$, define \[ \widehat{P_Kf}(\xi)=\mathbf1_{\{\abs\xi\geq K\}}\widehat f(\xi), \qquad \widehat{Q_Kf}(\xi)=\mathbf1_{\{\abs\xi<K\}}\widehat f(\xi). \]
On $\{\abs\xi\geq K\}$, one has $w_{s,I}^{\alpha,\beta}(\xi)\leq\langle K\rangle^{-\delta}w_{s+\delta,I}^{\alpha,\beta}(\xi)$. It follows from \eqref{eq:Abarron} that
\begin{equation} \label{eq:Bcontraction} \norm[\Bmix{s}{I}{\alpha,\beta}]{P_K\mathcal Au} \leq C_{\rm B}\langle K\rangle^{-\delta} \norm[\Bmix{s}{I}{\alpha,\beta}]{u}. \end{equation}
Similarly, \[ \norm[H^1]{P_Kf}^2 =\int_{\abs\xi\geq K}\langle\xi\rangle^2\abs{\widehat f(\xi)}^2\,\dd\xi \leq\langle K\rangle^{-2}\norm[H^2]{f}^2, \] and therefore
\begin{equation} \label{eq:Hcontraction} \norm[H^1]{P_K\mathcal Au} \leq C_{\rm H}\langle K\rangle^{-1}\norm[H^1]{u}. \end{equation}

Set \[ X_I=H^1(\R^{3N})\cap\Bmix{s}{I}{\alpha,\beta}(\R^{3N})\cap\mathfrak A_I, \qquad Y_I=H^1(\R^{3N})\cap\mathfrak A_I, \] and equip $X_I$ with \[ \norm[X_I]{u}=\norm[H^1]{u}+\norm[\Bmix{s}{I}{\alpha,\beta}]{u}. \]
Each antisymmetry constraint is the kernel of the continuous map $u\mapsto U_{ij}u+u$ on both ambient spaces, so $X_I$ and $Y_I$ are Banach spaces. The operators $V$, $\mathcal R_\mu$, $P_K$, and $Q_K$ commute with permutations of electron coordinates and hence preserve $\mathfrak A_I$. Define \[ \mathcal T_K=P_K\mathcal A, \qquad \gamma_K=\max\left\{C_{\rm B}\langle K\rangle^{-\delta}, C_{\rm H}\langle K\rangle^{-1}\right\}. \]

The estimates~\eqref{eq:Bcontraction} and \eqref{eq:Hcontraction} imply
\[ \norm[X_I]{\mathcal T_Ku}\leq \gamma_K\norm[X_I]{u}, \qquad \norm[H^1]{\mathcal T_Ku}\leq \gamma_K\norm[H^1]{u}.
\]
Choosing $K$ sufficiently large implies $\gamma_K<1$. Then
$\mathcal T_K$ is a strict contraction on both $X_I$ and $Y_I$.

The low-frequency part belongs to $X_I$. By Cauchy--Schwarz and Plancherel's theorem,
\begin{align}
  \norm[\Bmix{s}{I}{\alpha,\beta}]{Q_K\psi}
  &=\int_{\abs\xi<K}w_{s,I}^{\alpha,\beta}(\xi)\abs{\widehat\psi(\xi)}\,\dd\xi\notag\\
  &\leq\left(\int_{\abs\xi<K}\bigl(w_{s,I}^{\alpha,\beta}(\xi)\bigr)^2\,\dd\xi\right)^{1/2}
    \norm[L^2]{\psi}
  =C_{K,s,\alpha,\beta}\norm[L^2]{\psi}.
  \label{eq:lowfrequency}
\end{align}
Moreover, $Q_K\psi\in H^1\cap\mathfrak A_I$. By \eqref{eq:Abarron} and the $H^1$-to-$H^2$ estimate, \[ \mathcal A(Q_K\psi)\in H^2\cap\Bmix{s+\delta}{I}{\alpha,\beta}\cap\mathfrak A_I\subset X_I. \] Hence $g_K=P_K\mathcal A(Q_K\psi)\in X_I$.

Applying $P_K$ to \eqref{eq:fixedpoint}, we obtain
\[ P_K\psi=\mathcal T_K(P_K\psi)+g_K. \]
Since $\norm[\mathcal L(X_I)]{\mathcal T_K}<1$, the Neumann series \[ v=(\operatorname{Id}-\mathcal T_K)^{-1}g_K=\sum_{n=0}^\infty \mathcal T_K^ng_K \] converges in $X_I$ and gives the unique solution of $v=\mathcal T_Kv+g_K$ in that space. The function $P_K\psi\in Y_I$ satisfies the same equation. If $h_1,h_2\in Y_I$ are two solutions, then \[ \norm[H^1]{h_1-h_2} =\norm[H^1]{\mathcal T_K(h_1-h_2)} \leq \gamma_K\norm[H^1]{h_1-h_2}. \]
Since $\gamma_K<1$, one has $h_1=h_2$. Thus $P_K\psi=v\in X_I$, which together with~\eqref{eq:lowfrequency} proves Theorem~\ref{thm:mainI}.
\begin{remark}[Redundancy of the same-block bounds]
The resolvent estimate for pairs contained in $I$ also requires $\delta<1-\alpha$ and $\delta<2-s-2\alpha$. These inequalities impose no additional restriction on Theorem~\ref{thm:mainI}. Indeed, \eqref{eq:mainconditions} implies $s+\alpha<1$, while Theorem~\ref{thm:gain} permits \[ \delta<1-s-A_I\leq1-s-\alpha\leq1-\alpha, \qquad 1-s-\alpha<2-s-2\alpha. \]
\end{remark}

For Theorem~\ref{thm:spinrefined}, \eqref{eq:spin-antisymmetry} gives $\psi\in\mathfrak A_\sigma$. Fix $I\in\mathcal I_\sigma$, replace Theorem~\ref{thm:gain} by Theorem~\ref{thm:spingain}, and use \[ X_{I,\sigma}=H^1(\R^{3N})\cap\Bmix{s}{I}{\alpha,\beta}(\R^{3N})\cap\mathfrak A_\sigma, \qquad Y_\sigma=H^1(\R^{3N})\cap\mathfrak A_\sigma. \]
All operators used above preserve $\mathfrak A_\sigma$. The estimates \eqref{eq:Abarron}--\eqref{eq:Hcontraction}, the low-frequency estimate, and the two uniqueness arguments remain valid in these spaces, giving $\psi\in\Bmix{s}{I}{\alpha,\beta}$. Since \eqref{eq:spinconditions} is independent of $I$, this holds for every $I\in\mathcal I_\sigma$. This proves~\eqref{eq:spinintersection}.
\begin{proof}[Proof of Corollary~\ref{cor:FLinterpolation}]
  Theorem~\ref{thm:spinrefined} gives $w_{s_0,I}^{\alpha_0,\beta_0}\widehat\psi\in L^1$.
Under the conversion $\xi=2\pi\zeta$ from Meng's Fourier-transform convention to ours, the mixed-Sobolev estimate \cite[(2.3)--(2.7) and Corollary~2.4]{Meng2023} gives $w_{1,I}^{\alpha_2,\beta_2}\widehat\psi\in L^2(\R^{3N})$. Moreover, \[ w_{s_\theta,I}^{\alpha_\theta,\beta_\theta} =\bigl(w_{s_0,I}^{\alpha_0,\beta_0}\bigr)^{1-\theta} \bigl(w_{1,I}^{\alpha_2,\beta_2}\bigr)^\theta, \qquad \frac1p=(1-\theta)+\frac\theta2. \]
Since $1<p<2$, the exponents $(2-p)^{-1}$ and $(p-1)^{-1}$ are conjugate. We employ H\"older's inequality to obtain
\[
  \begin{aligned} \int_{\R^{3N}}\bigl(w_{s_\theta,I}^{\alpha_\theta,\beta_\theta}\abs{\widehat\psi}\bigr)^p\dd\xi
    &\leq
      \left(\int_{\R^{3N}}w_{s_0,I}^{\alpha_0,\beta_0}\abs{\widehat\psi}\,\dd\xi\right)^{p(1-\theta)}\\
    &\quad\times
      \left(\int_{\R^{3N}}\bigl(w_{1,I}^{\alpha_2,\beta_2}\abs{\widehat\psi}\bigr)^2\dd\xi\right)^{p\theta/2}.
  \end{aligned}
  \]
Taking the $p$th root, we prove the assertion.
\end{proof}

\section{Proofs of the sharpness results}\label{sec:consequences}

We first convert a nonvanishing Coulomb cusp into failure of the mixed Fourier $L^1$ norm at the boundary exponents and then apply this obstruction to the atomic states in Proposition~\ref{prop:atomicsharpness} and Corollary~\ref{cor:uniformsharpness}.

\begin{lemma}[Localized cusp obstruction]\label{lem:cuspobstruction}
  Let $I\subseteq\{1,\dots,N\}$ be nonempty, let $s,\alpha,\beta\geq0$, and let $u\in\Sch'(\R^{3N})$. Suppose that an open set $\mathcal U\subset\R^3\times\R^{3N-3}$ meets $\{r=0\}$ and that \[ u(r,y)=u_0(r,y)+\abs r u_1(r,y), \qquad (r,y)\in\mathcal U, \] where $u_0$ and $u_1$ are real analytic on $\mathcal U$. If $u_1(0,\cdot)\not\equiv0$ on $\{y\mid(0,y)\in\mathcal U\}$, then the following assertions hold.
  \begin{enumerate}[label=\textup{(\roman*)}]
    \item If $r=x_i-R_\nu$ is a nuclear collision variable, then $u\notin\Bmix{s}{I}{\alpha,\beta}$ whenever $s+a_i\geq1$.
    \item If $r=x_i-x_j$ is an electron--electron collision variable, then $u\notin\Bmix{s}{I}{\alpha,\beta}$ whenever $s+a_i+a_j\geq1$.
  \end{enumerate}
\end{lemma}

\begin{proof}
  The inequality $\langle z+z'\rangle\leq\sqrt2\langle z\rangle\langle z'\rangle$ implies \[ w_{s,I}^{\alpha,\beta}(\xi) \leq 2^{(s+\sum_{i=1}^Na_i)/2} w_{s,I}^{\alpha,\beta}(\xi-\zeta) w_{s,I}^{\alpha,\beta}(\zeta). \]
  For $\varphi\in C_c^\infty(\R^{3N})$, the Fourier product formula and Tonelli's theorem therefore yield
  \begin{align} \norm[\Bmix{s}{I}{\alpha,\beta}]{\varphi u} &\leq (2\pi)^{-3N/2}2^{(s+\sum_{i=1}^Na_i)/2} \norm[L^1]{w_{s,I}^{\alpha,\beta}\widehat\varphi} \norm[\Bmix{s}{I}{\alpha,\beta}]{u}. \label{eq:localizationbound} \end{align}
  Since $\widehat\varphi\in\Sch(\R^{3N})$, one has $\norm[L^1]{w_{s,I}^{\alpha,\beta}\widehat\varphi}<\infty$.

  Put $m=3N-3$ and let $(\rho,\eta)\in\R^3\times\R^m$ denote the Fourier variables dual to $(r,y)$. We write $A\Subset B$ if $\overline A$ is a compact subset of $B$. Assume first that $u_1(0,\cdot)\not\equiv0$ and choose $y_0\in\R^m$ such that $(0,y_0)\in\mathcal U$ and $u_1(0,y_0)\neq0$. There are $0<\varepsilon'<\varepsilon$ and open sets $y_0\in U_y'\Subset U_y\subset\R^m$ such that $B_\varepsilon(0)\times U_y\Subset\mathcal U$. Choose \[ \begin{aligned} &\chi_0\in C_c^\infty(B_\varepsilon(0)), &&0\leq\chi_0\leq1, &&\chi_0\equiv1\ \text{on }B_{\varepsilon'}(0),\\ &\chi_y\in C_c^\infty(U_y), &&0\leq\chi_y\leq1, &&\chi_y\equiv1\ \text{on }U_y'. \end{aligned} \] Define $\chi(r,y)=\chi_0(r)\chi_y(y)$ and $b_0(y)=\chi_y(y)u_1(0,y)$. Then \[ \chi\in C_c^\infty(\mathcal U), \qquad 0\leq\chi\leq1, \qquad \chi\equiv1\ \text{on }B_{\varepsilon'}(0)\times U_y', \qquad \supp\chi\Subset\mathcal U, \] while $b_0\in C_c^\infty(U_y)$ and $b_0(y_0)=u_1(0,y_0)\neq0$.

  Set $g(r)=\chi_0(r)\abs r$. Since $\chi_0\equiv1$ near the origin, the distributional identity $\Delta_r^2\abs r=-8\pi\delta_0$ gives \[ G\coloneqq\Delta_r^2g+8\pi\delta_0\in C_c^\infty(\R^3). \] Under the unitary Fourier convention, \[ \abs\rho^4\widehat g(\rho)=-8\pi(2\pi)^{-3/2}+\widehat G(\rho). \] Since $\widehat G$ is Schwartz, for every $M>0$, \[ \widehat g(\rho)=-2\sqrt{\frac2\pi}\,\abs\rho^{-4}+O_M(\abs\rho^{-M}) \qquad (\abs\rho\to\infty), \] and the same rapidly decaying remainder estimate holds after any number of $\rho$-derivatives.

  To justify the Taylor step, set \[ b_k(y)=\chi_y(y)\partial_{r_k}u_1(0,y), \qquad C_{k\ell}(r,y)=\chi_y(y)\int_0^1(1-\lambda)\partial_{r_k}\partial_{r_\ell}u_1(\lambda r,y)\,\dd\lambda. \] Since $B_\varepsilon(0)$ is star-shaped with respect to the origin, Taylor's formula with integral remainder gives, on $\supp\chi_0\times\supp\chi_y$, \[ \chi_y(y)u_1(r,y)=b_0(y)+\sum_{k=1}^3r_kb_k(y)+\sum_{k,\ell=1}^3r_kr_\ell C_{k\ell}(r,y). \] Consequently, \[ \chi u=\chi u_0+g(r)b_0(y)+\sum_{k=1}^3r_kg(r)b_k(y)+g(r)\sum_{k,\ell=1}^3r_kr_\ell C_{k\ell}(r,y). \]

  Since $\supp\chi\Subset\mathcal U$, extension by zero gives $\chi u_0\in C_c^\infty(\R^{3+m})$. Hence, for every $M_0\geq0$, there is a constant $C_{M_0}$ such that \[ \abs{\widehat{\chi u_0}(\rho,\eta)}\leq C_{M_0}\langle(\rho,\eta)\rangle^{-M_0}, \qquad (\rho,\eta)\in\R^3\times\R^m. \]
  Moreover, \[ \widehat{r_kg}(\rho)=i\partial_{\rho_k}\widehat g(\rho) =8i\sqrt{\frac2\pi}\,\rho_k\abs\rho^{-6}+O_M(\abs\rho^{-M}) \qquad (\abs\rho\to\infty). \]
  For the last sum in the decomposition above, a further first-order Taylor formula gives \[ C_{k\ell}(r,y)=C_{k\ell}(0,y)+\sum_{j=1}^3r_jE_{k\ell j}(r,y), \qquad E_{k\ell j}(r,y)=\int_0^1\partial_{r_j}C_{k\ell}(\lambda r,y)\,\dd\lambda. \]
  The terms containing $C_{k\ell}(0,y)$ satisfy \[ \widehat{r_kr_\ell g}(\rho)=-\partial_{\rho_k}\partial_{\rho_\ell}\widehat g(\rho) =O(\abs\rho^{-6}) \qquad (\abs\rho\to\infty). \]
  Define the remaining part by \[ h(r,y)\coloneqq g(r)\sum_{k,\ell,j=1}^3r_kr_\ell r_jE_{k\ell j}(r,y). \]
  Since $g(r)=\chi_0(r)\abs r$, this is a finite sum of terms $a_{k\ell j}(r,y)\abs r\,r_kr_\ell r_j$ with $a_{k\ell j}=\chi_0E_{k\ell j}\in C_c^\infty(\R^{3+m})$. Homogeneity gives, for every multi-index $\gamma$ with $\abs\gamma\leq6$, \[ \left|\partial_r^\gamma\bigl(\abs r\,r_kr_\ell r_j\bigr)\right| \leq C_\gamma\abs r^{4-\abs\gamma}, \qquad 0<\abs r\leq1. \]
  At order six the right-hand side is $C_\gamma\abs r^{-2}$, and \[ \int_{\{\abs r<1\}}\abs r^{-2}\,\dd r=4\pi<\infty. \]
  Integration by parts on $\{\abs r>\varepsilon\}$ produces boundary terms of order at most $O(\varepsilon)$, which vanish as $\varepsilon\downarrow0$. Thus these locally integrable classical derivatives represent the weak derivatives. The Leibniz rule gives $\partial_r^\gamma h\in L^1(\R^{3+m})$ for $\abs\gamma\leq6$. For each $\rho\neq0$, choose $k_0\in\{1,2,3\}$ such that $\abs{\rho_{k_0}}\geq\abs\rho/\sqrt3$. Six integrations by parts in $r_{k_0}$ give \[ \abs{\widehat h(\rho,\eta)} \leq (2\pi)^{-(3+m)/2}\abs{\rho_{k_0}}^{-6} \norm[L^1(\R^{3+m})]{\partial_{r_{k_0}}^6h} \leq C\abs\rho^{-6}, \] uniformly in $\eta$. Thus $\widehat h(\rho,\eta)=O(\abs\rho^{-6})$ uniformly in $\eta$.

  Combining these estimates yields the more precise expansion \begin{equation}\label{eq:cuspfourier} \begin{aligned} \widehat{\chi u}(\rho,\eta)={}&-2\sqrt{\frac2\pi}\,\abs\rho^{-4}\widehat {b_0}(\eta)\\ &+8i\sqrt{\frac2\pi}\,\abs\rho^{-6}\sum_{k=1}^3\rho_k\widehat {b_k}(\eta)+O(\abs\rho^{-6}) \qquad (\abs\rho\to\infty), \end{aligned} \end{equation} uniformly in $\eta\in\R^m$. In particular, the second line is $O(\abs\rho^{-5})$.

  Fourier injectivity and $b_0\not\equiv0$ give a point $\eta_0\in\R^m$ with $\widehat b_0(\eta_0)\neq0$. By continuity, there is $\varepsilon_\eta>0$ such that, with \[ \mathcal O=B_{\varepsilon_\eta}(\eta_0), \qquad b_*=\frac12\abs{\widehat b_0(\eta_0)}, \] one has $\abs{\widehat b_0(\eta)}\geq b_*$ for every $\eta\in\mathcal O$. Denote $m$-dimensional Lebesgue measure by $\mathcal L^m$. The uniform remainder in \eqref{eq:cuspfourier} yields a $\rho_{\mathcal O}\geq1$ for which
  \begin{equation} \label{eq:cusplowerbound} \abs{\widehat{\chi u}(\rho,\eta)} \geq\sqrt{\frac2\pi}\,b_*\abs\rho^{-4}, \qquad \eta\in\mathcal O,\quad \abs\rho\geq \rho_{\mathcal O}. \end{equation}

  Consider case (ii). Write \[ R=\frac{x_i+x_j}{2},\qquad r=x_i-x_j, \qquad y=(R,x_{\widehat{i,j}}),\qquad x_{\widehat{i,j}}=(x_k)_{k\notin\{i,j\}}. \]
  The corresponding Fourier variables are \[ \eta=(\eta_R,\eta')=(\xi_i+\xi_j,\xi_{\widehat{i,j}}), \qquad \rho=\frac{\xi_i-\xi_j}{2},\qquad \xi_{\widehat{i,j}}=(\xi_k)_{k\notin\{i,j\}}, \] and the active-variable transformation has
  \[
    \begin{pmatrix}\eta_R\\ \rho\end{pmatrix}
    =\begin{pmatrix}1&1\\[1mm] \frac12&-\frac12\end{pmatrix}
      \begin{pmatrix}\xi_i\\ \xi_j\end{pmatrix},
    \qquad
    \left|\det\begin{pmatrix}1&1\\ \frac12&-\frac12\end{pmatrix}\right|=1.
  \]
  The same matrix acts in each spatial component. Hence \[ \xi_i=\rho+\frac{\eta_R}{2},\qquad \xi_j=-\rho+\frac{\eta_R}{2},\qquad \xi_{\widehat{i,j}}=\eta',\qquad \dd\xi=\dd\rho\,\dd\eta. \]
  These relations define $\xi=\xi(\rho,\eta)$. Thus $\eta$ is the Fourier variable dual to the coordinates tangent to the collision manifold $\{r=0\}$. Set
  \[
  \begin{aligned}
    M_{\mathcal O}&=\sup_{\eta\in\mathcal O}\abs\eta,
    &\rho_0&=\max\{\rho_{\mathcal O},2M_{\mathcal O},1\},\\
    \mathcal C_{\mathcal O}&=\{(\rho,\eta)\in\R^3\times\R^m\mid\eta\in\mathcal O,\ \abs\rho\geq \rho_0\}.
  \end{aligned}
  \]
  For $(\rho,\eta)\in\mathcal C_{\mathcal O}$, \[ \frac34\abs\rho\leq\abs{\xi_i},\abs{\xi_j}\leq\frac54\abs\rho, \qquad \abs{\xi_i}^2+\abs{\xi_j}^2 =2\abs\rho^2+\frac12\abs{\eta_R}^2, \qquad \langle\xi\rangle\geq\abs\rho. \]
  Since all exponents in \eqref{eq:weight} are nonnegative, the full weight satisfies
  \begin{equation} \label{eq:pairweightlower} w_{s,I}^{\alpha,\beta}(\xi(\rho,\eta)) \geq \langle\xi\rangle^s\langle\xi_i\rangle^{a_i}\langle\xi_j\rangle^{a_j} \geq\left(\frac34\right)^{a_i+a_j}\abs\rho^{s+a_i+a_j} \end{equation}
  on $\mathcal C_{\mathcal O}$. Combining \eqref{eq:cusplowerbound}, \eqref{eq:pairweightlower}, and spherical coordinates in $\rho$ gives
  \[
  \begin{aligned}
    \norm[\Bmix{s}{I}{\alpha,\beta}]{\chi u}
    &\geq \int_{\mathcal O}\int_{\abs\rho\geq \rho_0}
      w_{s,I}^{\alpha,\beta}(\xi(\rho,\eta))
      \abs{\widehat{\chi u}(\rho,\eta)}\,\dd\rho\,\dd\eta\\
    &\geq4\pi\mathcal L^m(\mathcal O)\sqrt{\frac2\pi}\,b_*
      \left(\frac34\right)^{a_i+a_j}
      \int_{\rho_0}^\infty \varrho^{s+a_i+a_j-2}\,\dd\varrho.
  \end{aligned}
  \]
  The last integral diverges if and only if $s+a_i+a_j\geq1$, proving (ii).

  In case (i), take $r=x_i-R_\nu$, $y=x_{\widehat i}=(x_k)_{k\neq i}$, $\rho=\xi_i$, and $\eta=\xi_{\widehat i}=(\xi_k)_{k\neq i}$. The translation by $R_\nu$ contributes only the unimodular factor $e^{-iR_\nu\cdot\rho}$, and $\dd\xi=\dd\rho\,\dd\eta$. On $\mathcal O\times\{\abs\rho\geq \rho_{\mathcal O}\}$, one has $w_{s,I}^{\alpha,\beta}(\xi)\geq\langle\xi\rangle^s\langle\xi_i\rangle^{a_i}\geq\abs\rho^{s+a_i}$.
  Hence \[ \norm[\Bmix{s}{I}{\alpha,\beta}]{\chi u} \geq4\pi\mathcal L^m(\mathcal O)\sqrt{\frac2\pi}\,b_* \int_{\rho_{\mathcal O}}^\infty\varrho^{s+a_i-2}\,\dd\varrho, \] which is infinite if and only if $s+a_i\geq1$. This proves (i).

  If $u$ belonged to $\Bmix{s}{I}{\alpha,\beta}$, the localization bound \eqref{eq:localizationbound} would imply $\chi u\in\Bmix{s}{I}{\alpha,\beta}$, contradicting the divergent lower bounds in (i)--(ii). This completes the proof.
\end{proof}

The localized obstruction yields the endpoint tests in Proposition~\ref{prop:atomicsharpness}. The ground state $\Phi_{N,Z}$ has a continuous representative by \cite[Theorem~C.1.1]{Simon1982}; the Harnack inequality \cite[Theorem~C.1.3]{Simon1982} makes this representative strictly positive.

\begin{proof}[Proof of Proposition~\ref{prop:atomicsharpness}]
  The Fourier transform of $\psi_{\rm H}$ is a nonzero multiple of $(1+\abs\xi^2)^{-2}$; see \cite[(1.5)--(1.6)]{Yserentant2026}. Hence
  \[
  \begin{aligned}
    \norm[\Bmix{s}{\{1\}}{\alpha,\beta}]{\psi_{\rm H}}
    &=C\int_{\R^3}\langle\xi\rangle^{s+\alpha-4}\,\dd\xi\\
    &=4\pi C\int_0^\infty r^2(1+r^2)^{(s+\alpha-4)/2}\,\dd r.
  \end{aligned}
  \]
The integrand is $\mathcal{O}(r^2)$ as $r\downarrow0$ and is comparable to $r^{s+\alpha-2}$ as $r\to\infty$. The integral is finite if and only if $s+\alpha<1$, proving (i).

Fix a point at which $x_i=x_j$ and no other electron--electron or electron--nucleus collision occurs. Set $r=x_i-x_j$, $R=(x_i+x_j)/2$, and $y=(R,x_{\widehat{i,j}})$, where $x_{\widehat{i,j}}=(x_k)_{k\notin\{i,j\}}$. Restrict to a neighborhood meeting no other collision set. In this neighborhood, the local structure result \cite[Theorem~1.4 and Remark~1.6]{FournaisEtAl2009} gives
\[ \Phi_{N,Z}(r,y)=u_0(r,y)+\abs r u_1(r,y), \]
with analytic coefficients. In these coordinates,
\[ -\frac12(\Delta_{x_i}+\Delta_{x_j})=-\frac14\Delta_R-\Delta_r. \]
  All potential terms other than the isolated pair potential are locally bounded in these coordinates. On every smaller compact set in the tangential variable $y$, as $r\to0$ one has uniformly
  \[
  \begin{aligned}
    -\Delta_r\bigl(\abs r u_1(r,y)\bigr)
      &=-\frac{2u_1(0,y)}{\abs r}+O(1),\\
    \frac1{\abs r}\Phi_{N,Z}(r,y)
      &=\frac{u_0(0,y)}{\abs r}+O(1),
  \end{aligned}
  \]
  All remaining terms in $(H-E)\Phi_{N,Z}$ are $O(1)$, uniformly on such compact sets. Thus \[ (H-E)\Phi_{N,Z}=\frac{-2u_1(0,y)+u_0(0,y)}{\abs r}+O(1), \] so \[ u_1(0,y)=\frac12u_0(0,y). \]
  The continuous representative of $\Phi_{N,Z}$ is strictly positive, so $u_0(0,y)=\Phi_{N,Z}(0,y)>0$. Lemma~\ref{lem:cuspobstruction} applied to the pair $(1,2)$ proves the necessity of $s+\alpha+\beta<1$ in (ii). For (iii), the pairs $(1,2)$ and $(2,3)$ give $s+\alpha+\beta<1$ and $s+2\beta<1$, respectively. Conversely, antisymmetry with respect to the singleton $I=\{1\}$ is a void condition, and Theorem~\ref{thm:mainI} proves sufficiency in all three cases.
\end{proof}

\begin{proof}[Proof of Corollary~\ref{cor:uniformsharpness}]
  Theorem~\ref{thm:mainI} proves sufficiency. Proposition~\ref{prop:atomicsharpness} proves necessity for the three cases by taking, respectively, hydrogen with $(N,I)=(1,\{1\})$, helium with $(N,I)=(2,\{1\})$, and lithium with $(N,I)=(3,\{1\})$. Hence no boundary face of the family can be enlarged uniformly.
\end{proof}

\section{Conclusion}\label{sec:conclusion}

We established sharp mixed spectral Barron regularity for eigenfunctions of clamped-nuclei Coulomb Hamiltonians. Weighted Fourier $L^1$ estimates for the nuclear and electron--electron Coulomb operators, together with antisymmetric cancellation on same-spin blocks, yield index-set-dependent regularity and simultaneous product-moment bounds over all occupied spin blocks. Interpolation with known mixed-Sobolev estimates gives an intermediate Fourier--Lebesgue scale for $1<p<2$, while the hydrogen, helium, and lithium ground states establish the uniform sharpness of every defining inequality of the corresponding parameter regions.

Lemma~\ref{lem:sharp-isotropic-comparison} shows that, when $\alpha>0$ and $n_\sigma>1$, the mixed product-weighted space $\mathcal X_\sigma^{0,\alpha,0}$ is strictly contained in $\B^\alpha(\R^{3N})$ and thus captures additional regularity that is invisible to the isotropic Barron scale. Another natural question is whether neural-network architectures adapted to this product structure can exploit the additional regularity to achieve sharper approximation rates or lower parameter complexity for a prescribed accuracy than those obtained from isotropic Barron regularity alone.

In the isotropic setting, recent results show that quotients obtained by extracting suitable cut-off Jastrow factors belong to $\B^s$ for every $s<2$, and that this range is sharp \cite{Ehrlacher2026,MingYu2026Sharp}. A natural question is whether combining cusp extraction with the product weights introduced here yields a corresponding sharp mixed spectral Barron theory.
\appendix
\section{Space-comparison lemmas}\label{app:comparison-proofs}

This appendix contains the proofs of the two space-comparison results used in Section~\ref{sec:setting}. We first prove Lemma~\ref{lem:sharp-isotropic-comparison}, and then state and prove the two-spin result used after Theorem~\ref{thm:spinrefined}.

\begin{proof}[Proof of Lemma~\ref{lem:sharp-isotropic-comparison}]
  For positive continuous weights $U$ and $V$,
  \begin{equation}\label{eq:weighted-FL1-criterion} \mathcal F L^1(U)\hookrightarrow\mathcal F L^1(V) \quad\Longleftrightarrow\quad V(\xi)\leq CU(\xi)\quad(\xi\in\R^{3N}) \end{equation}
  for some $C>0$. The pointwise bound implies the embedding. Conversely, assume the embedding estimate, fix $\eta\in\R^{3N}$, and choose a nonnegative $\varphi\in C_c^\infty(\R^{3N})$ with $\int_{\R^{3N}}\varphi(\xi)\,\dd\xi=1$. Set \[ g_\varepsilon(\xi)=\varepsilon^{-3N}\varphi\!\left(\frac{\xi-\eta}{\varepsilon}\right), \qquad f_\varepsilon=\mathcal F^{-1}g_\varepsilon. \]
  Since $f_\varepsilon\in\Sch(\R^{3N})$, the embedding estimate gives \[ \int_{\R^{3N}}V(\xi)g_\varepsilon(\xi)\,\dd\xi \leq C\int_{\R^{3N}}U(\xi)g_\varepsilon(\xi)\,\dd\xi. \]
  Letting $\varepsilon\downarrow0$ and using continuity yields $V(\eta)\leq CU(\eta)$, proving \eqref{eq:weighted-FL1-criterion}.

  Write \[ W_{\sigma,\alpha}(\xi)=\sum_{I\in\mathcal I_\sigma}p_I(\xi)^\alpha, \qquad p_I(\xi)=\prod_{i\in I}\langle\xi_i\rangle. \]
  This is the weight defining the sum norm of $\mathcal X_\sigma^{0,\alpha,0}$. Since $\langle\xi_i\rangle\leq\langle\xi\rangle$, \[ W_{\sigma,\alpha}(\xi) \leq\sum_{I\in\mathcal I_\sigma}\langle\xi\rangle^{\alpha\abs I} \leq\abs{\mathcal I_\sigma}\langle\xi\rangle^{\alpha n_\sigma}. \]
  Hence $\B^t\hookrightarrow\mathcal X_\sigma^{0,\alpha,0}$ whenever $t\geq\alpha n_\sigma$.

  Since $\mathcal I_\sigma$ partitions $\{1,\dots,N\}$,
  \[
  \begin{aligned}
    \langle\xi\rangle^\alpha
    &=\left(1+\sum_{I\in\mathcal I_\sigma}\sum_{i\in I}\abs{\xi_i}^2\right)^{\alpha/2}\\
    &\leq\left(\sum_{I\in\mathcal I_\sigma}p_I(\xi)^2\right)^{\alpha/2}
    \leq\sum_{I\in\mathcal I_\sigma}p_I(\xi)^\alpha
    =W_{\sigma,\alpha}(\xi),
  \end{aligned}
  \]
  where the second inequality uses the subadditivity of $r\mapsto r^{\alpha/2}$. For $\alpha=0$, the conclusion follows from $1\leq\abs{\mathcal I_\sigma}=W_{\sigma,0}$. Thus $\mathcal X_\sigma^{0,\alpha,0}\hookrightarrow\B^t$ whenever $0\leq t\leq\alpha$.

  We prove the necessity of both conditions. Choose $I_*\in\mathcal I_\sigma$ with $\abs{I_*}=n_\sigma$ and a unit vector $e\in\R^3$. For $R>0$, let $\xi_i^{(R)}=Re$ if $i\in I_*$ and $\xi_i^{(R)}=0$ otherwise. Then \[ W_{\sigma,\alpha}(\xi^{(R)})=(1+R^2)^{\alpha n_\sigma/2}+\abs{\mathcal I_\sigma}-1, \qquad \langle\xi^{(R)}\rangle^t=(1+n_\sigma R^2)^{t/2}. \]
  If $\B^t\hookrightarrow\mathcal X_\sigma^{0,\alpha,0}$, criterion \eqref{eq:weighted-FL1-criterion} implies $W_{\sigma,\alpha}\leq C\langle\cdot\rangle^t$. The preceding ray then forces $t\geq\alpha n_\sigma$.

  Next fix $k\in\{1,\dots,N\}$ and let $\xi_k^{(R)}=Re$ and $\xi_i^{(R)}=0$ for $i\neq k$. Along this ray, \[ W_{\sigma,\alpha}(\xi^{(R)})=(1+R^2)^{\alpha/2}+\abs{\mathcal I_\sigma}-1, \qquad \langle\xi^{(R)}\rangle^t=(1+R^2)^{t/2}. \]
  If $\mathcal X_\sigma^{0,\alpha,0}\hookrightarrow\B^t$, the criterion implies $\langle\xi\rangle^t\leq CW_{\sigma,\alpha}(\xi)$ and hence $t\leq\alpha$. This proves both equivalences in \eqref{eq:sharp-isotropic-comparison}.

  If $\alpha>0$ and $n_\sigma>1$, the same two rays exclude the reverse of each endpoint embedding. All the weighted Fourier $L^1$ spaces involved are Banach spaces; equality as sets would therefore make the inverse identity continuous by the open mapping theorem. Both endpoint inclusions are consequently strict. If $\alpha=0$ or $n_\sigma=1$, the two endpoint bounds above give $W_{\sigma,\alpha}\asymp\langle\xi\rangle^\alpha$, and the corresponding spaces coincide with equivalent norms. This completes the proof.
\end{proof}

\begin{lemma}[Minimal mixed Barron space for two spin states]
  \label{lem:minimal-two-spin-space}
  Suppose that $\sigma\in\{1,2\}^N$, allowing either spin block to be empty.\footnote{If one spin block is empty, the exponent $\beta$ does not enter the defining weight, and the same minimality conclusion remains valid.} For $s,\alpha,\beta\geq0$, set $\tau=s+\alpha+\beta$. Then
  \begin{equation}\label{eq:minimal-two-spin-space} \mathcal X_\sigma^{0,\tau,0}\hookrightarrow\mathcal X_\sigma^{s,\alpha,\beta}, \qquad \norm[\mathcal X_\sigma^{s,\alpha,\beta}]{f} \leq\abs{\mathcal I_\sigma}N^{s/2}\norm[\mathcal X_\sigma^{0,\tau,0}]{f}. \end{equation}
  Thus $\mathcal X_\sigma^{0,\tau,0}$ is the least member, under continuous inclusion, of the family with fixed total order $\tau$. If both spin blocks are occupied, then $\mathcal X_\sigma^{0,\tau,0}=\mathcal X_\sigma^{0,0,\tau}$ with equal norms.
\end{lemma}

\begin{proof}[Proof of Lemma~\ref{lem:minimal-two-spin-space}]
  Suppose first that only one spin block $I_1$ is occupied, and set $p_1(\xi)=\prod_{i\in I_1}\langle\xi_i\rangle$. Then $I_1=\{1,\dots,N\}$ and $\langle\xi\rangle\leq N^{1/2}p_1(\xi)$. Since $s+\alpha\leq\tau$ and $p_1(\xi)\geq1$, \[ \langle\xi\rangle^sp_1(\xi)^\alpha \leq N^{s/2}p_1(\xi)^{s+\alpha} \leq N^{s/2}p_1(\xi)^\tau. \]
  Integration against $\abs{\widehat f}$ proves \eqref{eq:minimal-two-spin-space} in this case.

  Suppose next that both spin blocks $I_1,I_2$ are occupied, and set $p_\ell(\xi)=\prod_{i\in I_\ell}\langle\xi_i\rangle$ for $\ell=1,2$. Since the two blocks partition $\{1,\dots,N\}$, \[ \langle\xi\rangle \leq N^{1/2}\max\{p_1(\xi),p_2(\xi)\}. \]
  Using $s+\alpha+\beta=\tau$ gives \[ \langle\xi\rangle^s \bigl(p_1(\xi)^\alpha p_2(\xi)^\beta+p_2(\xi)^\alpha p_1(\xi)^\beta\bigr) \leq2N^{s/2}\max\{p_1(\xi),p_2(\xi)\}^\tau \leq2N^{s/2}\sum_{\ell=1}^2p_\ell(\xi)^\tau. \]
  Integration proves \eqref{eq:minimal-two-spin-space}. Since the triple $(0,\tau,0)$ belongs to the fixed-total-order family and its space embeds into every other member, it is the least member under continuous inclusion.

  When both blocks are occupied, the weights of $\mathcal X_\sigma^{0,\tau,0}$ and $\mathcal X_\sigma^{0,0,\tau}$ are, respectively, $p_1(\xi)^\tau+p_2(\xi)^\tau$ and $p_2(\xi)^\tau+p_1(\xi)^\tau$, so they coincide. When only one block is occupied, no complementary-coordinate weight occurs. This completes the proof.
\end{proof}

\begingroup
\small
\setlength{\bibsep}{2pt plus 0.5pt}
\setlength{\emergencystretch}{2em}
\sloppy
\hbadness=2000
\bibliographystyle{sn-mathphys-num}
\bibliography{references}
\endgroup

\end{document}